\documentclass[12pt]{amsart}
\usepackage[utf8]{inputenc}
\usepackage{setspace}
\usepackage[english]{babel}
\usepackage{mathabx}
\usepackage{framed}
\usepackage{soul}
\usepackage[all]{xy}
\usepackage[margin=1in]{geometry}
\usepackage{arydshln}
\usepackage[colorinlistoftodos]{todonotes}
\usepackage{pb-diagram}
\usepackage[mathscr]{euscript}
\usepackage{multicol}
\usepackage{stmaryrd} 
\usepackage{empheq} 
\usepackage{tikz-cd}
\tikzset{
  symbol/.style={
    draw=none,
    every to/.append style={
      edge node={node [sloped, allow upside down, auto=false]{$#1$}}}
  }
}

\usepackage{amsmath,amsfonts,amssymb,amsthm,mathrsfs,latexsym,mathtools}
\usepackage{commath}
\usepackage[T1]{fontenc}
\usepackage{hyperref}
\usepackage{caption} 
\usepackage{subcaption} 

\usepackage{lipsum}

\DeclareMathAlphabet{\mathbbmsl}{U}{bbm}{m}{sl}

\usetikzlibrary{matrix}
\title{Hilbert Scheme of A Pair of Skew Lines on Cubic Threefolds}

\newcommand{\C}{\mathbb{C}} 
\newcommand{\Z}{\mathbb{Z}}

\newcommand{\BP}{\mathbb{P}}

\newcommand{\Bl}{\textup{Bl}}

\newcommand{\Sym}{\textup{Sym}}

\newtheorem{theorem}{Theorem}[section]
\newtheorem{definition}[theorem]{Definition}
\newtheorem{proposition}[theorem]{Proposition}
\newtheorem{corollary}[theorem]{Corollary}
\newtheorem{question}[theorem]{Question}
\newtheorem{example}[theorem]{Example}
\newtheorem{lemma}[theorem]{Lemma}
\newtheorem{conjecture}[theorem]{Conjecture}
\newtheorem{problem}[theorem]{Problem}
\newtheorem{remark}[theorem]{Remark}

\newtheorem{thm}{Theorem}

\author{Yilong Zhang}
\address{Department of Mathematics\\
Purdue University\\
  150 N University St, West Lafayette, IN 47907, USA}
\email{zhan4740@purdue.edu}

\date{Feb 27, 2025}
\subjclass[2010]{14C05, 14J30 primary}
\keywords{Hilbert scheme, cubic threefold, a pair of skew lines, ADE singularities, Bridgeland moduli space}

\begin{document}
\maketitle
\begin{abstract}
Two general lines on a smooth cubic threefold $X$ are disjoint and determine an irreducible component of the Hilbert scheme of $X$. We prove that this component is smooth and isomorphic to the Hilbert scheme of two points of the Fano varieties of lines of $X$. We also study its relation to the geometry of lines and singularities on the hyperplane sections of $X$ and its relation to Bridgeland moduli spaces.
\end{abstract}

\section{introduction}
Let $X\subseteq \mathbb P^4$ be a smooth cubic threefold. Let $L_1, L_2$ be a pair of skew lines on $X$, i.e., $L_1\cap L_2=\emptyset$, then they span a hyperplane in $\mathbb P^4$. Let $F$ be the parameter space of lines in $X$. Then there is a rational map
\begin{equation}\label{eqn_Intro_rationalmap}
    \Phi: F\times F\dashrightarrow (\mathbb P^4)^*,\ (L_1,L_2)\mapsto \textup{Span}(L_1,L_2).
\end{equation}

This map is generically finite since when the two lines are general, the hyperplane section $X\cap \textup{Span}(L_1,L_2)$ is a cubic surface and contains 27 lines.

Clemens and Griffiths \cite[Lem. 12.16]{CG} showed that this map is still regular when the two lines intersect at one point $L_1\cap L_2=\{p\}$. In this case, $\Phi$ associates the tangent hyperplane $T_pX$ at $p$. They also analyzed the limiting hyperplane for a one-parameter family through the diagonal, where two lines coincide \cite[Lem. 12.17]{CG}.

In fact, as a consequence of a result of Beauville \cite{Beauville}, $\Phi$ extends to a regular map on the blow-up of the diagonal $\Delta_F$
\begin{equation}\label{eqn_Intro_blowup}
    \tilde{\Phi}: \Bl_{\Delta_F}(F\times F)\to (\mathbb P^4)^*.
\end{equation}
Hence, we can say that $L\in F$ together with a normal direction $v\in H^0(N_{L|X})$ determines a unique hyperplane and can be thought of as a pair of "skew lines". 

The first motivation for this paper is the following.
\begin{question}\label{Intro_Q1}
    Is there a modular interpretation of the map \eqref{eqn_Intro_blowup}?
\end{question}

We will provide an answer using the Hilbert scheme. This is inspired by the work of Chen, Coskun, and Nollet \cite{CCN}, where they studied the component of Hilbert scheme of projective space containing a pair of skew lines.

\subsection{Hilbert Scheme of a Pair of Skew Lines}

 If $L_1$ and $L_2$ are disjoint, then the union $L_1\cup L_2$ can be regarded as a closed (reduced) subscheme of $X$ with Hilbert polynomial $2n+2$. According to Grothendieck \cite{Groth}, there is a projective scheme $\textup{Hilb}^{2n+2}(X)$ parameterizing the universal family of closed subschemes of $X$ with Hilbert polynomial $2n+2$. Then, by forgetting the order, there is a 2-to-1 rational map
\begin{equation}\label{Intro_eqn_2-to-1}
    F\times F \dashrightarrow \textup{Hilb}^{2n+2}(X),\ (L_1,L_2)\mapsto \mathcal{O}_{L_1\cup L_2},
\end{equation}
 whose image is a Zariski open dense subspace of an irreducible component $H(X)$ of $\textup{Hilb}^{2n+2}(X)$. The component $H(X)$ parameterizes flat families of subschemes of $X$ whose general member is a pair of skew lines (see Section \ref{Subsection_2.HS_Pn} for details). We call $H(X)$ the \textit{Hilbert scheme of a pair of skew lines} of $X$. 
 
By classification of flat limits of a pair of skew lines in projective spaces (cf. Section \ref{Subsection_2.HS_Pn}), $H(X)$ parameterizes four types of subschemes of $X$ with Hilbert polynomial $2n+2$. A type (I) scheme is a pair of skew lines; A type (III) scheme is supported on a pair of incidental lines with an embedded point at the intersection; a type (II)/(IV) scheme is generically nonreduced and is supported on a double line.

Let $\textup{Sym}^2F=(F\times F)/\Z_2$ denote the symmetric square of $F$. There is a Hilbert-Chow morphism \cite[Thm. 6.3]{K-RationalCurves}
\begin{equation}\label{eqn_Intro_HilbertChow}
    H(X)\to \Sym^2F,
\end{equation}
which is birational, and is an isomorphism over the locus where pairs of lines are disjoint. 

 Our main result is to describe this birational morphism.

\begin{theorem} (cf. Theorem \ref{main_theorem})\label{IntroHschemeThm}
$H(X)$ is smooth and isomorphic to the Hilbert scheme of two points $F^{[2]}$. 
\end{theorem}

In particular, the Hilbert-Chow morphism \eqref{eqn_Intro_HilbertChow} is the blow-up of the diagonal. This is analogous to the results by Chen, Coskun, and Nollet \cite{CCN}, where they proved the smoothness of the component $H(\mathbb P^3)$ of a pair of skew lines in the Hilbert scheme of $\mathbb P^3$ is smooth and isomorphic to the blow-up of the symmetric square of $Gr(2,4)$ on the diagonal. When $m\ge 4$, $H(\mathbb P^m)$ arises from a two-step blowup on the Chow variety. Especially we need their result for $m=4$ to prove Theorem \ref{main_theorem}. This is because $X\subseteq \mathbb P^4$ is a closed subscheme, hence $H(X)$ is a closed subscheme of $H(\mathbb P^m)$.

Another input is the geometry of the normal bundle $N_{L|X}$ of a line to the cubic threefold. It turns out that the line with normal bundle $\mathcal{O}\oplus \mathcal{O}$ has a $\mathbb P^1$-family of quadric surfaces tangent to it and corresponds to a $\mathbb P^1$-family of type (II) schemes; A line with normal bundle $\mathcal{O}(-1)\oplus \mathcal{O}(1)$ has a unique plane tangent to it, and by varying the embedded point, there is a $\mathbb P^1$-family of type (IV) schemes supported on it. Moreover, the parameterization is continuous as the line of the normal bundle $\mathcal{O}\oplus \mathcal{O}$ specializes to $\mathcal{O}(-1)\oplus \mathcal{O}(1)$. This is governed by the geometry of the dual map $x\mapsto T_xX$ along a line.

\subsection{Relative Hilbert Scheme}

Just as a pair of skew lines span a hyperplane, each scheme $Z\in H(X)$ spans a hyperplane. In particular, there is a morphism
\begin{equation}\label{Intro_eqn_pi}
    \pi: H(X)\to (\mathbb P^4)^*.
\end{equation}

It endows $H(X)$ a relative Hilbert scheme with respect to the universal family of hyperplane sections of $X$. In particular, the fiber $\pi^{-1}(H)$ parameterizes the subschemes of type (I)-(IV) of cubic surfaces $X\cap H$.

As a consequence of Theorem \ref{IntroHschemeThm}, we answer Question \ref{Intro_Q1} by the following.

\begin{proposition} (cf. Corollary \ref{Cor_Q1Answer})  \label{Intro_Prop}
The blow-up  $\Bl_{\Delta_F}(F\times F)$ is isomorphic to the double cover of $H(X)$ branched along the exceptional divisor and parameterizes subschemes of $X$ of type (I)-(IV) up to an order. The map $\tilde{\Phi}$ factors through 
$$\pi: H(X)\to (\mathbb P^4)^*,$$
which associates each subscheme of four types of $X$ to the unique hyperplane it spans. 
\end{proposition}

For general cubic threefold $X$, $X\cap H$ has only ADE singularities \cite[Lem. 2.7]{YZ_Cubic3fold}, but for special $X$, the hyperplane section may be a cone over an elliptic curve $E$ — which has an elliptic singularity at the cone point. This happens exactly when $H=T_pX$ is the tangent hyperplane at an Eckardt point $p$, through which infinitely many lines on $X$ pass. We found that $\pi$ is flat over the locus where $X\cap H$ has only ADE singularities and has positive dimensional fiber over $H$, which is the tangent hyperplane of an Eckardt point of $X$ (cf. Proposition \ref{HYtoP4dual}).

\begin{proposition} (cf. Proposition \ref{HYtoP4dual})
 $\pi$ is generically finite of degree 216, and the positive dimensional fiber corresponds to an Eckardt point of $X$ and is isomorphic to $\Sym ^2E$, where $E$ is an elliptic curve. 
\end{proposition}

\subsection{Relation to Bridgeland Moduli Spaces}

 According to \cite{APR}, there is a Bridgeland stable moduli space $\mathcal{M}_{\sigma}(w)$ whose general point parameterizes coherent sheaf $i_*\mathcal{O}_S(L-M)$ where $S$ a smooth hyperplane section of $X$ and $L,M$ a pair of skew lines and $i:S\to X$ is the inclusion. Let $JX$ denote the intermediate Jacobian of $X$. Then there is an Abel-Jacobi map $AJ:\mathcal{M}_{\sigma}(w)\to JX$ by taking the second Chern class, followed by the standard Abel-Jacobi map, which is described in \eqref{eqn_AJM_{sigma}}.

\begin{proposition}\label{IntroBridgeland}
Let $\Bl_{\Delta_F}(F\times F)\to JX$ be the composition of the blow-up and the Abel-Jacobi map. Then it factors through the Bridgeland moduli space $\mathcal{M}_{\sigma}(w)$ and makes the following diagram commute up to a translation on the torus.

\begin{figure}[ht]
    \centering
\begin{tikzcd}
\Bl_{\Delta_F}(F\times F)\arrow[dr] \arrow[r,"\psi"] &  \mathcal{M}_{\sigma}(w) \arrow[d,"AJ"]\\
  & JX.
\end{tikzcd}
\end{figure}{}

Moreover, $\psi$ has the following modular interpretation: it sends a general type (I) scheme with an order to $\mathcal{O}_S(L_1-L_2)$; it sends a type (II) and (IV) scheme $Z$ to the ideal sheaf $I_{p|S}$, where $p$ is the unique point on the support line of $Z$ determined by Corollary \ref{Cor_Q1Answer}. 
\end{proposition}

We also conjecture $\psi$ has modular interpretation for type (III) schemes and special type (I) schemes (cf. Conjecture \ref{Conj_typeIII}). We also proposed our next plan to extend Theorem \ref{IntroHschemeThm} to higher dimensions, particularly cubic fourfold. These will be explained in detail in Section \ref{Section_5.HS_Modular}.

\subsection*{Outline}
In Section \ref{Section_Cubic3fold_Prelim}, we will review basic facts on cubic surfaces and cubic threefolds. In Section \ref{Subsection_2.HS_Pn}, we will review the result of the Hilbert scheme of a pair of skew lines on projective spaces studied by \cite{CCN}. In Section \ref{Section_HS}, we study the Hilbert scheme $H(X)$ of a pair of skew lines on cubic threefold $X$ and prove Theorem \ref{IntroHschemeThm}. In Section \ref{Section_5.HS_CubicSurface}, we will study $H(X)$ restricted to a hyperplane section of $X$ and prove Proposition \ref{HYtoP4dual}. In Section \ref{Section_5.HS_Modular}, we will discuss some relationships to Bridgeland moduli spaces and a problem on cubic fourfold.

\subsection*{Acknowledgement} I would like to thank my advisor, Herb Clemens, for introducing me to this topic, answering my questions, and for his constant encouragement. I would like to thank Izzet Coskun for helpful conversations on Hilbert schemes as well as his lectures on K3 surfaces during the pandemic. Besides, I would like to thank Arend Bayer, Laure Flapan, Ritvik Ramkumar, Franco Rota, Will Sawin, and Shizhuo Zhang for many useful communications. In addition, I would like to thank the referees for their suggestions to revise the paper.

\section{Preliminaries}\label{Section_Cubic3fold_Prelim}

\subsection{Lines on Cubic Threefold}
Let $X$ be a smooth hypersurface of $\mathbb P^4$ of degree three. Let $Gr(2,5)$ be the Grassmanian of the lines in $\mathbb P^4$. The set of lines contained in $X$ is naturally a subvariety $F$ of $Gr(2,5)$ and is called the \textit{Fano variety of lines} of $X$. $F$ is a smooth surface of general type and parameterizes two types of lines distinguished by normal bundles. One can refer to \cite[Chapter 5]{Huy} for a detailed discussion on the related topics about cubic threefolds.

\begin{definition}\label{Def_Line12type_cubic3fold}
\normalfont A line $L\subseteq X$ is called to be of the \textit{first type} if the normal bundle $N_{L|X}\cong \mathcal{O}_L\oplus \mathcal{O}_L$; $L$ is called to be of the \textit{second type} if the normal bundle $N_{L|X}\cong \mathcal{O}_L(1)\oplus \mathcal{O}_L(-1)$.
\end{definition}

Equivalently, the dual map along the line $L$
\begin{equation}\label{DualMap}
    \mathcal{D}_{|L}:L\to (\mathbb P^4)^*,\ x\mapsto T_xX.
\end{equation}
is 1-to-1 onto a conic if $L$ is the first type if $\mathcal{D}_{|L}$, and is 2-to-1 onto a line if $L$ is the second type (cf. \cite[Def. 6.6]{CG}).

The lines of the first type are generic in $F$, while the set of lines of the second type forms a divisor of $F$.

\begin{proposition}  \label{prop_line12type_Prelim}
Let $L$ be a line on $X$, then

(1) $L$ is of the first type iff there is a smooth quadric surface in $\mathbb P^4$ tangent to $X$ along $L$.

(2) $L$ is of the second type iff there is a unique plane $P_L\cong \mathbb P^2$ tangent to $X$ along $L$.
\end{proposition}

\begin{proof} These are essentially due to Clemens and Griffiths (cf. \cite[Lem. 6.7, 6.18 (i)]{CG}). To give a quick geometric reason, suppose $L$ is of the first type and given by equation $L=\{x_2=x_3=x_4=0\}$, then the equation of $X$ can be expressed as \cite[(6.9)]{CG}
\begin{equation}\label{eqn_Lqstequation}
   x_2x_0^2+x_3x_0x_1+x_4x_1^2+\textup{higher order terms in}~x_2,x_3,x_4. 
\end{equation}

Remove the higher order terms and set $x_2=0$, it has a factor a quadric equation
\begin{equation} \label{eqn_quadricsurface1}
   x_3x_0+x_4x_1=0.
\end{equation}

When $L$ is of the second type, the image of the dual map \eqref{DualMap} is a line. So the $\mathbb P^2=\cap_{x\in L}T_xX$ is a plane tangent to $X$ along $L$.

Conversely, the normal bundle of a line in the quadric surface and $\mathbb P^2$ is a subbundle of the normal bundle $N_{L|X}$ and uniquely determines it.
\end{proof}

\begin{definition}\normalfont
An \textit{Eckardt point} $p$ on a cubic threefold $X$ is a point through which infinitely many lines on $X$ pass. 
\end{definition}

For a general point on $X$, there are 6 lines through it, so Eckardt points are special. There are at most finitely many Eckardt points on a smooth cubic threefold \cite[Lem. 10.15]{CG}. 

\begin{proposition}\label{Prop_Eckardt}
   The following statements are equivalent:
\begin{itemize}
    \item[(1)] $p\in X$ is an Eckardt point.
    \item[(2)] Lines on $X$ through $p$ form an elliptic curve.
    \item[(3)] The tangent hyperplane section $T_pX\cap X$ of $X$ at $p$ is a cone over a smooth plane cubic curve.
    \item[(4)] The hyperplane $H$ tangents to $X$ at $p$ and $H\cap X$ has an elliptic singularity. 
    
\end{itemize}  
\end{proposition}
\begin{proof}
$(1)\leftrightarrow(2)\leftrightarrow(3)$ is due to \cite[Lem. 8.1]{CG}. $(3)\leftrightarrow(4)$ is by classification of normal cubic surface (cf. Lem. \ref{classicication of cubic surfaces}). 
\end{proof}

Consequently, an Eckardt point $p$ corresponds to an elliptic curve $E\subseteq F$. We will unspokenly use this correspondence throughout the paper.

\subsection{Abel-Jacobi Map}
The intermediate Jacobian $JX$ of a smooth cubic threefold $X$ is a principally polarized abelian variety of dimension 5. Clemens and Griffiths \cite{CG} showed that $JX$ is not a product of Jacobian of curves using the Abel-Jacobi map  \cite{CG}
\begin{equation}
    \Psi: F\times F\to JX. \label{AJ}
\end{equation}

It is generically 6-to-1 onto the theta divisor $\Theta$ and contracts the diagonal $\Delta_F$ to a singular point $0$ of $\Theta$.

Beauville \cite{Beauville} showed that $0$ is the only singularity of $\Theta$ and the blow-up $\Bl_0(\Theta)$ is smooth, with the exceptional divisor isomorphic to the cubic threefold $X$. As a consequence, this provides an alternative proof of the Torelli theorem for cubic threefolds.

In addition, Beauville showed that the Abel-Jacobi map extends to the blowup on the both domain and the target and factors through the blow-up $\Bl_0(\Theta)$.

\begin{lemma}\cite{Beauville}\label{Lemma_BlTheta}
 There is a commutative diagram

\begin{figure}[ht]
    \centering
    \begin{equation}\label{diagram_blowupAJ}
\begin{tikzcd}
\textup{Bl}_{\Delta_F}(F\times F)\arrow[d,"\sigma"] \arrow[r,"\tilde{\Psi}"] &  \textup{Bl}_0(\Theta) \arrow[d]\arrow[r,hookrightarrow] & \textup{Bl}_0JX\arrow[d]\\
F\times F\arrow[r,"\Psi"]&\Theta \arrow[r,hookrightarrow]& JX.
\end{tikzcd}
    \end{equation}
\end{figure}
\end{lemma}


The exceptional divisor of $\sigma$ is isomorphic to the projectivized tangent bundle $\mathbb PT_F$. In particular, the fiber $\sigma^{-1}(L,L)\cong \mathbb P^1$ over a double line on the diagonal is identified with $\mathbb PH^0(N_{L|X})$.

There is a geometric interpretation of the map $\tilde{\Psi}$ on the exceptional divisor using incidence correspondence described in \cite[Prop. 12.31]{CG}. Let $I$ be the subvariety of $F\times X$ together with the two projections
\begin{figure}[ht]
    \centering
    \begin{equation}\label{diagram_incidence}
\begin{tikzcd}
&I=\{(t,x)\in F\times X|x\in L_t\}\arrow[dr,"p_2"]\arrow[dl,"p_1"'] \\
F&& X.
\end{tikzcd}
    \end{equation}
\end{figure}

Then $p_1$ is a $\mathbb P^1$ bundle, and $p_2$ is generically 6-to-1.

\begin{lemma}(cf. \cite[p.342-p.345]{CG}, \cite{Tju}, \cite{AT}) \label{P^1_bundle_iso_lemma}
There is a canonical isomorphism 
\begin{equation}
  \alpha:\mathbb PT_F\cong I \label{P^1_bundle_iso}
\end{equation}
as $\mathbb P^1$-bundle over $F$.
\end{lemma}

\begin{lemma}(cf. \cite[p.244]{Huy})\label{Lemma_AJlift-on-diagonal}
Via the identification \eqref{P^1_bundle_iso}, the restriction of the extended Abel-Jacobi map $\tilde{\Psi}$ to the exceptional divisor $\sigma^{-1}(\Delta_F)$ is identified to the projection
$$p_2:I\to X,~(t,x)\mapsto x.$$
\end{lemma}

\subsection{Gauss Map}
Gauss map is the rational map
\begin{equation}\label{Eqn_GaussMap}
   \mathcal{G}:\Theta\dashrightarrow (\mathbb P^4)^*,\: p\mapsto T_p\Theta 
\end{equation}
that associates a smooth point $p$ to its projective tangent hyperplane $T_p\Theta$. Note that here we used the fact that the cotangent bundle of Jacobian is trivial.

\begin{lemma} \cite[13.6]{CG} \label{eqn_AJ-Gauss}
The rational map \eqref{eqn_Intro_rationalmap} factors through the Abel-Jacobi map \eqref{AJ} and Gauss map \eqref{Eqn_GaussMap}. In other words, there is a commutative diagram

\begin{figure}[ht]
    \centering
    \begin{equation}\label{Diagram_PhiPsiG}
\begin{tikzcd}
F\times F\arrow[dr, dashed, "\Phi"] \arrow[r,"\Psi"] &  \Theta \arrow[d,dashed,"\mathcal{G}"]\\
&(\mathbb P^4)^*.
\end{tikzcd}
    \end{equation}
\end{figure}
\end{lemma}

According to Lemma \ref{Lemma_BlTheta}, the Gauss map \eqref{Eqn_GaussMap} extends to a morphism 
\begin{equation}\label{Eqn_LiftGaussMap}
   \tilde{\mathcal{G}}:\Bl_0(\Theta)\to (\mathbb P^4)^*
\end{equation}

 The restriction of $\tilde{\mathcal{G}}$ to the exceptional divisor is the dual map $$X\to (\mathbb P^4)^*,\  x\mapsto T_xX.$$ This result, together with \eqref{diagram_blowupAJ}, will prove that \eqref{eqn_Intro_blowup} is indeed a morphism.

 \begin{corollary}
     $\Phi$ extends to a morphism $\tilde{\Phi}$ on the blow-up, and factors through $\tilde{\mathcal{G}}$. In particular, diagram \eqref{Diagram_PhiPsiG} extends to

\begin{figure}[ht]
    \centering
\begin{equation}\label{AJblowupDiagram}
\begin{tikzcd}
\Bl_{\Delta_F}(F\times F) \arrow[r,"\tilde{\Psi}"]\arrow[dr,"\tilde{\Phi}"] &  \textup{Bl}_0(\Theta) \arrow[d,"\tilde{\mathcal{G}}"]\\
& (\mathbb P^4)^*.
\end{tikzcd}
\end{equation}
\end{figure}

 \end{corollary}

\section{Hilbert Scheme of a Pair of Skew Lines on Projective Spaces}\label{Subsection_2.HS_Pn}

In this section, we will review the results of the Hilbert scheme of a pair of skew lines on the projective spaces by Chen, Coskun, and Nollet \cite{CCN}.

Let $m\ge 3$, the Hilbert scheme $\textup{Hilb}^{2n+2}(\mathbb P^m)$ has two irreducible components, $H(\mathbb P^m)$ and $H'(\mathbb P^m)$. A general point in $H(\mathbb P^m)$ parameterizes a pair of skew lines. A general point in $H'(\mathbb P^m)$ parameterizes the union of a smooth conic and an isolated point. 

\begin{definition}\normalfont\label{Def_Hm}
Call $H(\mathbb P^m)$ the \textit{Hilbert scheme of a pair of skew lines} on $\mathbb P^m$.
\end{definition}

$H(\mathbb P^m)$ parameterizes four types of schemes:
\begin{itemize}
    \item[(I)] A pair of skew lines;
    \item[(II)] A line with a double structure remembering the normal direction to a quadric surface;
    \item[(III)] A pair of incident lines with an embedded point at the intersection;
    \item[(IV)] A line with a double structure remembering the normal direction to a plane, together with an embedded point on the line.
\end{itemize}

\begin{figure}[h]
     \centering 
     \begin{subfigure}[b]{0.2\textwidth}
         \centering
         \includegraphics[width=\textwidth]{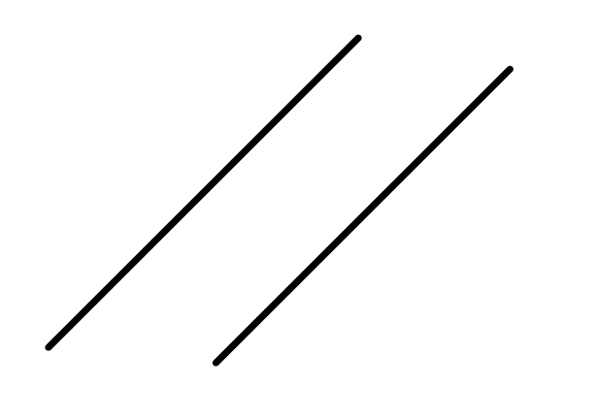}
         \caption*{Type (I)}
         \label{fig:Type (I)}
     \end{subfigure}
     \hspace{6em}
     \begin{subfigure}[b]{0.17\textwidth}
         \centering
         \includegraphics[width=\textwidth]{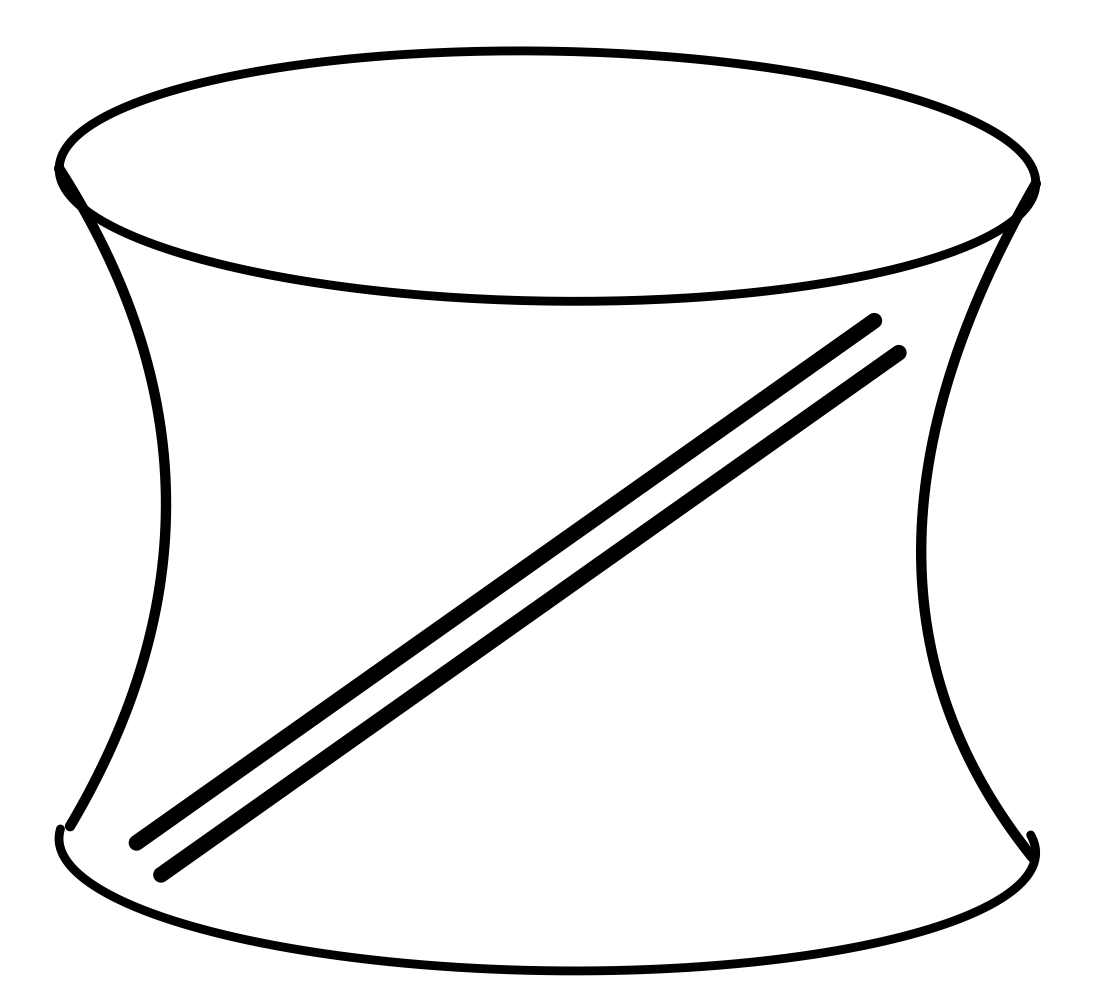}
         \caption*{Type (II)}
         \label{fig:Type (II)}
     \end{subfigure}\\
   \begin{subfigure}[b]{0.2\textwidth}
         \centering
         \includegraphics[width=\textwidth]{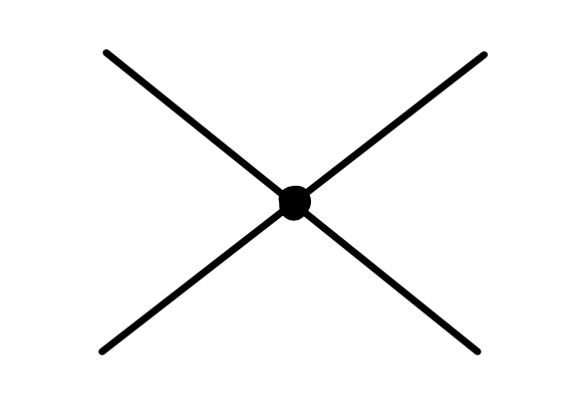}
         \caption*{Type (III)}
         \label{fig:Type (III)}
     \end{subfigure}
     \hspace{6em}
     \begin{subfigure}[b]{0.2\textwidth}
         \centering
         \includegraphics[width=\textwidth]{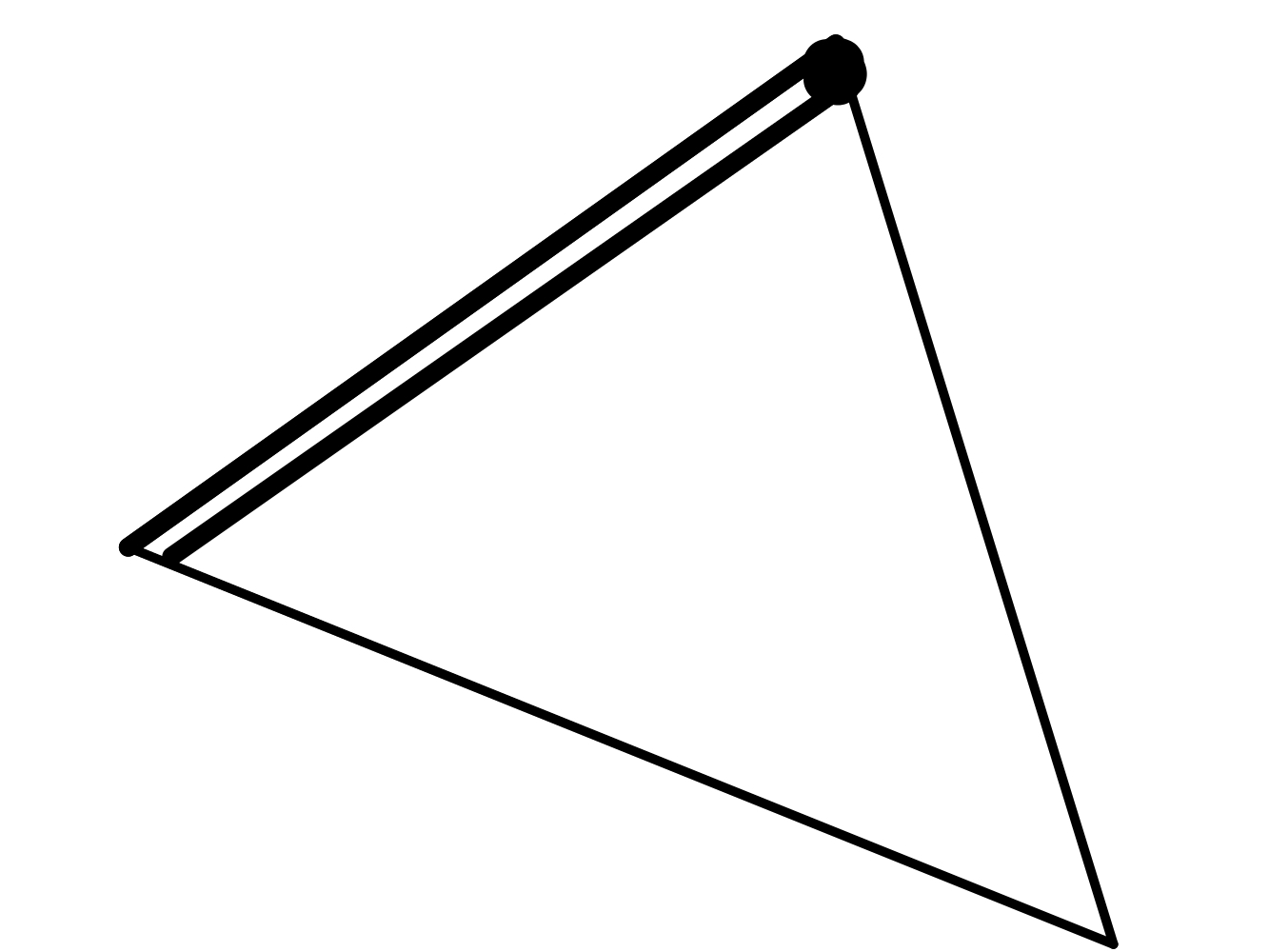}
         \caption*{Type (IV)}
         \label{fig:Type (IV)}
     \end{subfigure}
        \caption{Schemes of the Four Types}
        \label{fig:Type}
\end{figure}

In a scheme-theoretic language, a type (II) scheme is the first-order infinitesimal neighborhood of a line in a smooth quadric surface; a type (IV) scheme is the first-order infinitesimal neighborhood of a line in a plane, together with an embedded point on the line. The flat degenerations of the four types of schemes can be described in \cite[p.4]{CCN}.

Harris \cite[1.b]{Harris} asked if $H(\mathbb P^3)$ is smooth. The proof strategy is sketched in \cite[3.5.10]{Lee} by Lee and was carried out by Chen, Coskun, and Nollet in \cite{CCN} by investigating the deformation on the locus where two components $H(\mathbb P^{m})$ and $H'(\mathbb P^m)$ intersect. 

\begin{theorem}(\cite[Thm. 1.1, Cor. 2.8]{CCN})\label{thm_CCN}
When $m\ge 3$, the component $H(\mathbb P^m)$ is smooth.
\end{theorem}

When $m=3$, the Hilbert-Chow morphism
$$\rho_3: H(\mathbb P^3)\to \Sym^2Gr(2,4)$$
is the blow-up of the diagonal. In other words, $H(\mathbb P^3)$ is isomorphic to the Hilbert scheme of two points $Gr(2,4)^{[2]}$.

Just like two skew lines span a $\mathbb P^3$, each of the schemes of the four types (I)-(IV) is contained in a unique linear subspace $\mathbb P^3$ of $\mathbb P^m$ (Lem. 3.5.3, \cite{Lee}), hence there is a projection
\begin{equation}\label{eqn_piP^m}
  \pi: H(\mathbb P^m)\to Gr(4,m+1)  
\end{equation}
with fiber isomorphic to $H(\mathbb P^3)$, when $m\ge 4$. In addition, the Hilbert-Chow morphism 
\begin{equation}\label{Eqn_HC}
    \rho_m:H(\mathbb P^m)\to \textup{Sym}^2Gr(2,m+1)
\end{equation}
factors through $Gr(2,m+1)^{[2]}$ and becomes a two-step blow-up:

\begin{proposition}\label{H4_blowup}
The Hilbert-Chow morphism \eqref{Eqn_HC} factors through 
\begin{equation} \label{H4_blowup_formula}
    H(\mathbb P^m)\xrightarrow{\sigma_2}\textup{Bl}_{\Delta}\textup{Sym}^2Gr(2,m+1)\xrightarrow{\sigma_1} \textup{Sym}^2Gr(2,m+1),
\end{equation}

\noindent where $\sigma_1$ blows up the diagonal, and $\sigma_2$ blows up the strict transform of the locus $D\subseteq \textup{Sym}^2Gr(2,m+1)$ parameterizing pairs of incident lines.
\end{proposition}

\begin{proof}
Let $I_{\Delta}$ be the ideal sheaf of the diagonal $\Delta\subseteq \textup{Sym}^2Gr(2,m+1)$, then the pullback $p^*I_{\Delta}$ is an ideal sheaf of a divisor, which is invertible since $H(\mathbb P^m)$ is smooth (Thm. \ref{thm_CCN}). So by the universal property of the blowup (\cite{Hartshorne}, Prop. II.7.14), the Hilbert-Chow morphism \eqref{Eqn_HC} factors through $\textup{Bl}_{\Delta}\textup{Sym}^2Gr(2,m+1)$ as in $(\ref{H4_blowup_formula})$, where $\sigma_1$ blows up the diagonal and $\sigma_2$ is birational. 

Let $D\subseteq \textup{Sym}^2Gr(2,m+1)$ denote the locus of a pair of incident lines and $\tilde{D}$ the strict transform, which has codimension $m-2$. For a type (III) supported on a pair of incident lines $L_1\cup L_2$, the embedded point determines and is uniquely determined by a $\mathbb P^3$ containing $L_1\cup L_2$, and there is a $\mathbb P^{m-3}$-family of such hyperplanes. Therefore, the general fiber over $D$ (and therefore over $\tilde{D}$) is isomorphic to $\mathbb P^{m-3}$. Now, $\sigma_2^{-1}(\tilde{D})$ is a divisor. By the same argument, $\sigma_2$ factors through $W=\textup{Bl}_{\tilde{D}}\textup{Bl}_{\Delta}\textup{Sym}^2Gr(2,m+1)$. We have a commutative diagram

\begin{figure}[ht]
    \centering
\begin{tikzcd}
H(\mathbb P^m) \arrow[d,"\sigma_3"] \arrow[r, "\pi"] \arrow[dd, bend right, swap, "\rho_m"] & Gr(4,m+1)\\
W\arrow[ur, "\pi'"]\arrow[d]\\
\textup{Sym}^2Gr(2,m+1).  \arrow[uur,dashrightarrow,"\phi"] 
\end{tikzcd}
\end{figure}{}

$\pi'$ is the morphism induced by the rational map $\phi:(L_1,L_2)\mapsto \textup{span}(L_1,L_2)$ and $\pi=\pi'\circ\sigma_3$. The fiber of $\phi$ is a dense subset of $\textup{Sym}^2Gr(2,4)$, whose closure in $W$ is isomorphic to $\textup{Bl}_{\Delta}\textup{Sym}^2Gr(2,4)$, so it is the fiber of $\pi'$. On the other hand, we know that the fiber of $\pi$ is also $H_3\cong \textup{Bl}_{\Delta}\textup{Sym}^2Gr(2,4)$, so $\sigma_3$ is a bijective birational morphism. Therefore, by Zariski's main theorem, it is an isomorphism.
\end{proof}

One can find a similar argument in \cite[p.15]{CCN} and \cite[Prop. 6.11]{Ritvik}. The map $\sigma_2$ can be regarded as a map that "forgets" embedded points.

\begin{example}\normalfont We describe the Hilbert-Chow morphism for $m=4$ case
$$\rho_4: H(\mathbb P^4)\to \textup{Sym}^2Gr(2,5).$$

 The complement $\textup{Sym}^2Gr(2,5)\setminus D$ parameterizes type (I) schemes, where $\rho_4$ is an isomorphism; Over the locus $D\setminus \Delta$ where two lines interest at a point, $H(\mathbb P^4)$ is a $\mathbb P^1$-bundle: the fiber parameterizes the hyperplanes containing the two lines and corresponds to different type (III) schemes; Over the diagonal $\Delta$, $H(\mathbb P^4)$ is a $\tilde{\mathbb P^5}$-bundle, where $f:\tilde{\mathbb P^5}\to \mathbb P^5$ is blow-up along a codimension-two locus where strict transform of $D$ and $\mathbb P^5$ intersect (smooth and defined by three quadrics). The exceptional divisor of $f$ parameterizes type (IV) schemes, while the complement parameterizes type (II) schemes.
\end{example}

\section{Hibert Scheme of a Pair of Skew Lines on Cubic Threefolds}\label{Section_HS}

In this section, we characterize the irreducible component $H(X)$ of the Hilbert scheme of a cubic threefold $X$ that contains a pair of skew lines in $X$. Our main goal is to prove Theorem \ref{main_theorem}. As a consequence, Corollary \ref{Cor_Q1Answer} will provide a new interpretation of the morphism \eqref{eqn_Intro_blowup} and Lemma \ref{Lemma_AJlift-on-diagonal}.

\subsection{Main Theorem}


Recall that $F$ is the Fano surface of the lines in cubic threefold, then through the natural inclusion $F\hookrightarrow Gr(2,5)$, there is the following commutative diagram

\begin{figure}[ht]
    \centering
\begin{equation}\label{HilbertChow}
\begin{tikzcd}
H(X)\arrow[d,"\rho_X"] \arrow[r,hookrightarrow] &  H(\mathbb P^4) \arrow[d,"\rho_4"]\\
\textup{Sym}^2F\arrow[r,hookrightarrow]  & \textup{Sym}^2Gr(2,5),
\end{tikzcd}
\end{equation}
\end{figure}{}
\noindent where the vertical maps are Hilbert-Chow morphisms.

$\rho_4$ is a successive two-step blow-up according to Proposition \eqref{H4_blowup}. By the universal property of blow-up, $\rho_X$ is the strict transform. Our goal is to explicitly characterize $H(X)$ through this birational morphism.

\begin{theorem}\label{main_theorem}
 $H(X)$ is smooth and is isomorphic to $\textup{Bl}_{\Delta_F}\textup{Sym}^2F$, the blow-up of symmetric square of $F$ along the diagonal.
\end{theorem}

To interpret our result, there is a unique way of assigning a subscheme of type (III) of $X$ to a pair of incident lines. The type (II) and type (IV) locus form a $\mathbb P^1$-bundle over $F$, which is identified with the incidence variety $I$ (cf. Lemma \ref{P^1_bundle_iso_lemma}). Note that our result is analogous to $H(\mathbb P^3)\cong \Bl_{\Delta}Gr(2,4)$ (cf. Theorem \ref{thm_CCN}), but different from projective space where a double line can support both type (II) and type (IV) schemes, a double line $L$ in cubic threefold $X$ can only support a single type depending on the normal bundle $N_{L|X}$.

As a consequence of Theorem \ref{main_theorem}, there is a double cover $\widetilde{H(X)}$ of $H(X)$ branched along the exceptional divisor. It comes from a commutative diagram

\begin{figure}[h]
   \centering
\begin{tikzcd}
\widetilde{H(X)}=\textup{Bl}_{\Delta_F}(F\times F)\arrow[r,"/\mathbb Z_2"] \arrow[d,""] & \textup{Bl}_{\Delta_F}\textup{Sym}^2F\cong H(X) \arrow[d] \\
F\times F\arrow[r,"/\mathbb Z_2"]& \textup{Sym}^2F.
\end{tikzcd}
\end{figure}

Hence, to reinterpret Beauville's result that the rational map $F\times F\dashrightarrow (\mathbb P^4)^*$ extends to a morphism on the blow-up \eqref{eqn_Intro_blowup}, we can say the subscheme $Z\in H(X)$ of $X$ of the four types generalize the notion of "a pair of skew lines". We denote $\pi:H(X)\to (\mathbb P^4)^*$ the composition of inclusion $H(X)\hookrightarrow H(\mathbb P^4)$ and \eqref{eqn_piP^m}.

\begin{corollary}\label{Cor_Q1Answer}The map \eqref{eqn_Intro_blowup} factors through the composition 
$$\widetilde{H(X)}\to H(X)\xrightarrow{\pi} (\mathbb P^4)^*,$$
 and assigns each subscheme $Z\in H(X)$ of $X$ to the unique hyperplane that contains it. Moreover, its restriction to the exceptional divisor 
$$\tilde{\Psi}|_{\mathcal{E}}: \mathbb PT_F\to X$$
is given by the following: It sends a type (II)/(IV) scheme $Z$ to a point $x$ such that $T_xX$ is the hyperplane spanned by $Z$.
\end{corollary}

This also provides a new interpretation of Lemma \ref{Lemma_AJlift-on-diagonal}.

\begin{remark}\normalfont
    We can regard $\widetilde{H(X)}$ as the "Hilbert scheme" of a pair of ordered skew lines, or rather nested Hilbert scheme: There is a closed subscheme $\textup{Hilb}^{2n+2,n+1}(X)$ of $\textup{Hilb}^{2n+2}(X)\times \textup{Hilb}^{n+1}(X)$ that parameterizes a pair of subschemes $(Z_1,Z_2)$ of $X$ with $Z_2\subseteq Z_1$ and has Hilbert polynomials $p_n(Z_1)=2n+2$ and $p_n(Z_2)=n+1$. So there is an inclusion $\widetilde{H(X)}\hookrightarrow \textup{Hilb}^{2n+2,n+1}(X)$ and the image parameterizes a pair $(Z,L)$ with $Z\in H(X)$ and $L\subseteq Z$ a reduced line. The branched double cover $\widetilde{H(X)}\to H(X)$ is induced by projecting to the first coordinate.
\end{remark}

\begin{remark}\normalfont (Relation to Hilbert scheme of two points)
As a consequence of Theorem \ref{main_theorem}, the Hilbert scheme of a pair of skew lines $H(X)$ is isomorphic to the Hilbert scheme of two points $F^{[2]}$ on the Fano surface of lines $F$. Similarly, $\widetilde{H(X)}$ is isomorphic to the nested Hibert scheme of points $F^{[2,1]}$. 
\end{remark}

\subsection{Strategy of the Proof}
We prove Theorem \ref{main_theorem} in two steps.\\
\textbf{Step 1.} Construct a bijective morphism $\delta:\Bl_{\Delta_F}\Sym^2F\to H(X)$.\\
\textbf{Step 2.} Show $\delta$ is an isomorphism.

Note that our strategy is similar to the strategy to prove Theorem \ref{thm_CCN} in the case of projective spaces (cf. \cite[Thm 1.1]{CCN}). For them, the difficulty lies in Step 2 and requires careful study of deformation theory at a point where $H(\mathbb P^3)$ and the other component $H'(\mathbb P^3)$ intersect.  For us, however, Step 2 will follow from the smoothness of $H(\mathbb P^4)$. Instead, the hard work lies in Step 1. This requires a classification of all subschemes of $X$ of the four types on $X$ and an understanding of their deformation.


The rest of the sections are devoted to proving Theorem \ref{main_theorem}. We will first classify all subschemes of $X$ of type (I)-(IV). In particular, we will find a parameterization of type (II) and (IV) subschemes (Proposition \ref{Prop_type1_P1family}, \ref{Prop_type2_P1family}) using the geometry of normal bundle $N_{L|X}$. In \ref{subsection_Proof1}, we will complete Step 1 based on Beauville's result (cf. Lemma \ref{Lemma_BlTheta}). In \ref{subsection_Step2}, we will complete Step 2 by showing the smoothness of $H(X)$.

\subsection{Type (I) and type (III) schemes} \label{subsection_I,III}
Now we study the subschemes of $X$ parameterized by $H(X)$. First of all, type (I) schemes have the form $L_1\cup L_2$ when $L_1$ and $L_2$ are a pair of disjoint lines on $X$. For the type (III) scheme, we have the following result.
 
\begin{lemma} \label{typeIII_lemma} When $L_1$ and $L_2$ intersect at one point $x\in X$, there is a unique type (III) scheme $Z\in H(X)$ supported on $L_1\cup L_2$, and
$$Z=L_1\cup L_2\cup Z_{x,T_pX},$$
where $Z_{x,T_pX}$ is the scheme associated with the ideal $I_p^2$, where $I_p$ is the ideal on the reduced point $p$ in the hyperplane $T_pX$.

\end{lemma}
\begin{proof}
A type (III) subscheme $Z$ of $\mathbb P^4$ is a union $Z=Z_{\textup{red}}\cup Z_{x,H}$, where $Z_{\textup{red}}=L_1\cup L_2$ is the reduced scheme and $Z_{x,H}$ is an embedded point supported on $x$ and contained in a hyperplane $H\in (\mathbb P^4)^*$. Assume a pair of incidental lines $L_1\cup L_2$ is contained in $X$, the condition that $Z_{x,H}$ is a closed subscheme of $X$ is equivalent to the hyperplane $H$ is tangent to $X$ at $x$, i.e., $H=T_xX$.
\end{proof}

\begin{corollary}\label{HC_offdiagonal}
The Hilbert-Chow morphism $$\rho:H(X)\to \textup{Sym}^2F$$ is isomorphic over $\textup{Sym}^2F\setminus \Delta_F$.
\end{corollary}
Note that a similar result is obtained in \cite[Lemma 12.16]{CG} using the analytic method.
\begin{proof}
$\rho$ sends the set of type (I) and (III) schemes to $\textup{Sym}^2F\setminus \Delta_F$, which is bijective by Lemma \ref{typeIII_lemma}. So by Zariski's main theorem, it is an isomorphism.
\end{proof}

\subsection{Type (II) and Type (IV) schemes}\label{subsection_II,IV}
Next, we start to discuss the schemes parameterized by $H(X)$ and supported on a single line, that is, the type (II) and type (IV) schemes. It turns out that they correspond to the first and the second types of lines on cubic threefold.

The following is a consequence of Proposition \ref{prop_line12type_Prelim}.

\begin{lemma} \label{pure_type}
There is a correspondence:
$$\textup{line of first type}\Longleftrightarrow \textup{type (II) scheme},$$
$$\textup{line of second type}\Longleftrightarrow \textup{type (IV) scheme}.$$

In other words, if $Z\in H(X)$ is a scheme supported on a line $L$, then $L$ is of the first type if and only if $Z$ is of type (II); $L$ is of the second type if and only if $Z$ is of type (IV).
\end{lemma}

The following argument is a generalization of Proposition \ref{prop_line12type_Prelim}
\begin{proposition} \label{quadric_surfaces_lemma}
Let $L$ be a line of the first type on $X$. Then there is a $\mathbb P^1$-family of smooth quadric surfaces in $\mathbb P^4$ tangent to $X$ along $L$.

 More precisely, there is a one-to-one correspondence
\begin{align}\label{Eqn_TypeIIQuadric}
    \mathbb PH^0(N_{L|X})&\leftrightarrow L\leftrightarrow \text{smooth quadric surfaces tangent to}~X~\text{along}~L.\\
    v&\leftrightarrow x \leftrightarrow Q_{x}\subseteq T_xX\nonumber,
\end{align}
where $v\in \mathbb PH^0(N_{L|X})$ associates a hyperplane $H_v$ tangent to $X$ at a unique point $x\in L$.  $H_v=T_xX$ contains a quadric surface $Q_x$ such that $TQ_{x}|_{L}\subseteq TX|_{L}$.

\end{proposition} 
The idea is the following. Each section $v\in H^0(N_{L|X})$ can be identified with a line $L_v$ in $\mathbb P^4$ disjoint from $L$. By scaling $v$, we get a ruling of lines that swept out a quadric surface. In fact,  this is implicit in \cite[Lem. 6.18]{CG}. We provide a self-contained proof here.

\begin{proof}

We choose the equation of $L$ and $X$ as in the proof of Proposition \ref{prop_line12type_Prelim}. A section $v\in H^0(N_{L|X})$ corresponds to a line $L_{v}$ in $\mathbb P^4$ via
$$N_{L|X}\cong \mathcal{O}_{L}\oplus \mathcal{O}_L\hookrightarrow \mathcal{O}_L(1)\oplus \mathcal{O}_L(1)\oplus \mathcal{O}_L(1)\cong N_{L|\mathbb P^4}.$$

By linearizing the equation of $F$ in $Gr(2,5)$ \cite[(6.14)]{CG}, we find $L_{v}$ is parameterized by
\begin{equation}
    \lambda[1:0:0:a:-b]+\mu[0:1:-a:b:0],~[\lambda:\mu]\in \mathbb P^1, \label{LvParam}
\end{equation}
for some $a,b\in \mathbb C$ not both zero. 

$L$ and $L_v$ span a hyperplane 
\begin{equation} \label{eqn_Hv}
    H_v=a^2x_4+b^2x_2+abx_3=0.
\end{equation}
Note this is the tangent hyperplane $T_{[b:a]}X$ at point $[b:a]\in L$.

Now, let's find the equation of quadric surface swept out by \eqref{LvParam} as we scale $(a,b)$ by $s\in \C^*$. The line $L_{sv}$ satisfies three equations 
$$-sax_0-sbx_1+x_3=0,\ sax_1+x_2=0,\ sbx_0+x_4=0.$$

By canceling out the factor $s$, we find one linear equation \eqref{eqn_Hv} and two quadric equations.
\begin{equation}\label{quadric_surfaces}
    \begin{cases}
    Q_1=bx_1x_2+ax_1x_3+ax_0x_2=0,\\
    Q_2=ax_0x_4+bx_0x_3+bx_1x_4=0.
    \end{cases}
\end{equation}

Note that when $a$ and $b$ are both nonzero, $H_v=Q_1=0$ and $H_v=Q_2=0$ are equivalent. To see it defines a smooth quadric surface, for example, when $a\neq 0$, $Q_1$ defines a cone over a smooth quadric surface. Since $H_v=0$ does not pass through the cone point, it cuts a smooth quadric section. When $a=0$, \eqref{eqn_Hv} and $Q_2=0$ define the same quadric surfaces as in \eqref{eqn_quadricsurface1}. Finally, $H_v$ coincides with the tangent hyperplane $T_{[b:a]}X$ at point $[b,a]\in L$. So \eqref{eqn_Hv} and \eqref{quadric_surfaces} define a $\mathbb P^1$-family of smooth quadric surface in $\mathbb P^4$ tangent to $X$ along $L$.
\end{proof}

One should note that for each normal direction $v$, the choice of the quadric surface is not unique. (For example, one can wiggle the quadric surface \eqref{eqn_quadricsurface1} by adding an arbitrary quadratic term in $x_3$ and $x_4$.) However, they agree to the first order along $L$ with respect to the parameters in the normal direction. Our choices \eqref{Eqn_TypeIIQuadric} are canonical in the sense that they have vanishing second-order data. This suffices our need to accommodate type (II) schemes. 

\begin{proposition}\label{Prop_type1_P1family}
   Let $L$ be a line of the first type on $X$. Then the set of type (II) subschemes of $X$ supported on $L$ is parameterized by $\mathbb PH^0(N_{L|X})\cong \mathbb P^1$.
\end{proposition}
\begin{proof}
   The subscheme is to take first-order infinitesimal neighborhood $Z_{L,Q_x}$ of $L$ in the quadric surface $Q_x$, with $x\in L$. Conversely, any type (II) subscheme of $X$ supported on $L$ is represented by the first-order neighborhood $Z_{L,Q}$ of $L$ in a smooth quadric surface $Q$ tangent to $X$ along $L$. Then $Q$ spans a hyperplane, which is tangent to a point $x\in L$. Then $Q$ agrees with $Q_x$ to the first order. So $Z_{L,Q}$ coincides with $Z_{L,Q_x}$. \end{proof}

\begin{proposition}\label{Prop_type2_P1family}
    Let $L$ be a line of second type on $X$. Then the set of type (IV) subschemes of $X$ supported on $L$ are parameterized by $\mathbb PH^0(N_{L|X})$. 
\end{proposition}
\begin{proof}
   A line $L$ of the second type has a unique plane $\mathbb P$ tangent to $X$ along $L$ (cf. Proposition \ref{prop_line12type_Prelim} (2)). Then the first-order infinitesimal neighborhood $Z_{L,P}$ of $L$ in $\mathbb P$ is contained in $X$. By varying the embedded point $Z_{p,T_pX}$ supported on $p\in L$ as the first-order infinitesimal neighborhood in the tangent hyperplane $T_pX$, we have a $\mathbb P^1$-family of type (IV) subschemes $Z_{L,P}\cup Z_{p,T_pX}$ of $X$ supported on $L$. Note that $p\in L$ can be also parameterized by vanishing point of section $v\in \mathbb PH^0(\mathcal{O}(1))$. The first part of the claim follows.
   
   Conversely, since $Z_{L,P}$ is determined by the normal bundle $N_{L|X}$; the embedded point supported on $p\in L$ has to be contained in $T_pX$. Hence these are the only type (IV) subschemes.
\end{proof}

Therefore, the fiber of Hilbert-Chow morphism $H(X)\to \Sym^2F$ over a closed point on the diagonal is identified to $\mathbb P^1$, at least set theoretically.


\subsection{Proof for Step 1} \label{subsection_Proof1}
Recall that Beauville's result on the extended Abel-Jacobi map on the blow-ups \eqref{diagram_blowupAJ} implies $\Bl_{\Delta_F}(F\times F)\to (\mathbb P^4)^*$ is a morphism \eqref{AJblowupDiagram}. This provides a continuous way of assigning a hyperplane. This rigidifies the parameterization of type (II)/(IV) scheme over the diagonal.

\begin{proposition}\label{bijective}
There is a bijective morphism 
\begin{equation}\label{eqn_delta}
   \delta:\textup{Bl}_{\Delta_F}\textup{Sym}^2F\to H(X). 
\end{equation}
\end{proposition}

\begin{proof}
First, $\delta$ is rational and is defined on the locus of pairs of skew lines by assigning $(L_1,L_2)\mapsto \mathcal{O}_{L_1\cup L_2}$. As a consequence of Corollary \ref{HC_offdiagonal}, $\delta$ is regular and bijective on the complement of the exceptional divisor. We aim to show that it extends to a morphism.

Since $\tilde{\Phi}$ in \eqref{AJblowupDiagram} is independent of the order of the two lines, it factors through the $\mathbb Z_2$-quotient $$\tilde{\Phi}':\textup{Bl}_{\Delta_F}\textup{Sym}^2F\to (\mathbb P^4)^*$$ and provides a continuous way to assign hyperplanes. 

Denote $\pi:H(X)\to (\mathbb P^4)^*$ the natural projection. Then we have the composition $\pi\circ\delta=\tilde{\Phi}'$ is regular. Using the properties of the dual map along lines of cubic threefold \eqref{DualMap}, there are at most two points on the exceptional divisor supported on the same line and has the same image under $\tilde{\Phi}'$, indicating that the projection $\overline{\Gamma(\delta)}\to \textup{Bl}_{\Delta_F}\textup{Sym}^2F$ from the graph closure $\overline{\Gamma(\delta)}$ of $\delta$ is finite, and thus an isomorphism by Zariski's main theorem. It follows that $\delta$ is a morphism. Finally, injectivity can be checked on each fiber of the exceptional divisor over $\Delta_F$ using Proposition \ref{Prop_type1_P1family} and \ref{Prop_type2_P1family}.
\end{proof}

\begin{remark}\normalfont
In fact, one can show that the parameterization of the type (II) scheme in Proposition \ref{Prop_type1_P1family} and the type (IV) scheme in Proposition \ref{Prop_type2_P1family} (up to taking a "conjugation") are continuous, with respect to the degeneration of the normal bundle $\mathcal{O}\oplus\mathcal{O}$ of line of the first type to $\mathcal{O}(1)\oplus\mathcal{O}(-1)$ of line of the second type \cite[P. 144-148]{YZ_thesis}. This provides an alternative proof of Proposition \ref{bijective} without invoking the regularity of \eqref{AJblowupDiagram}. 
\end{remark}

\subsection{Proof for Step 2}
\label{subsection_Step2}
Now, we are ready to prove the main theorem. Note that the bijective morphism \eqref{eqn_delta} will be an isomorphism if $H(X)$ is normal, as a consequence of Zariski's main theorem. However, this is not known a priori. (This is also the difficulty of proving the smoothness of $H(\mathbb P^3)$ in \cite{CCN}.) In principle, it is possible that $H(X)$ is highly singular and $\delta$ is the normalization map, e.g. the normalization of a cuspidal curve or the dual map of a smooth hypersurface. Note that the tangent map of these examples fails to be of maximal rank somewhere.

For us, it suffices to prove that the tangent map of $\delta$ is everywhere of maximal rank, or equivalently, $\delta$ is an immersion. We will use the fact that $H(X)$ is a closed subscheme of $H(\mathbb P^4)$, where the latter is smooth.\\

\noindent\textit{Proof of Theorem \ref{main_theorem}.}
It suffices to show that the bijective morphism \eqref{eqn_delta} is an isomorphism. By the discussion above, it suffices to show that $\delta$, composed with the closed immersion $i:H(X)\hookrightarrow H(\mathbb P^4)$ is an immersion, namely, of maximal rank at each point.

By expressing $H(\mathbb P^4)\to \textup{Sym}^2Gr(2,5)$ as successive blow-ups as in Proposition \ref{H4_blowup}, we have the following commutative diagram.

\begin{figure}[ht]
    \centering
\begin{equation*} 
\begin{tikzcd}
& H(\mathbb P^4) \arrow[d,"\sigma_2"] \\
\textup{Bl}_{\Delta_F}\textup{Sym}^2F \arrow[ur,"i\circ\delta"] \arrow[r,"\phi"] \arrow[d] & \textup{Bl}_{\Delta}\textup{Sym}^2Gr(2,5) \arrow[d,"\sigma_1"]\\
\textup{Sym}^2F\arrow[r]&\textup{Sym}^2Gr(2,5)
\end{tikzcd}
\end{equation*}
\end{figure}{}
$\phi$ is the unique map that extends the rational map $\textup{Sym}^2F\dashrightarrow \textup{Bl}_{\Delta}\textup{Sym}^2Gr(2,5)$ induced by the universal property of the blow-up (Cor. II.7.15, \cite{Hartshorne}). Also note that the two sides of $\phi$ are the Hilbert schemes of two points $F^{[2]}$ on $F$, and $Gr(2,5)^{[2]}$ on $Gr(2,5)$, respectively, so $\phi$ is identified with the inclusion of smooth varieties $$F^{[2]}\hookrightarrow Gr(2,5)^{[2]}.$$ 

It follows that $\sigma_2\circ i\circ \delta=\phi$ is an immersion, so $i\circ\delta$ has to be an immersion.
\qed

\section{Restriction to a Hyperplane Section}\label{Section_5.HS_CubicSurface}

In the last section, we characterize the Hilbert scheme of a pair of skew lines $H(X)$ on a smooth cubic threefold $X$. We studied the projection map 
\begin{equation}\label{eqn_piH(X)->P4*}
    \pi: H(X)\to (\BP^4)^*,
\end{equation}
which sends $Z\in H(X)$ to the unique hyperplane that contains $Z$.

Note that this provides an interpretation of $H(X)$ as a relative Hilbert scheme. Let $\mathcal{X}=\{(x,t)\in X\times (\mathbb P^4)^*|x\in X\cap H_t\}$ and $pr_i$ the projection to the $i$-th coordinate. Then $pr_2:\mathcal{X}\to (\mathbb P^4)^*$
is the universal family hyperplane sections of $X$ with the relative ample bundle $pr_1^*(\mathcal{O}(1))$. By Grothendieck, there is a projective scheme $\textup{Hilb}^{2n+2}(\mathcal{X}/(\mathbb P^4)^*)$ over $(\mathbb P^4)^*$ and parameterizes the universal flat family of subschemes of $X$ over $(\mathbb P^4)^*$ with Hilbert polynomial $2n+2$. Our first observation is the following
\begin{proposition}\label{Prop_RelHS}
   The map \eqref{eqn_piH(X)->P4*} defines $H(X)$ as an irreducible component of the relative Hilbert scheme $\textup{Hilb}^{2n+2}(\mathcal{X}/(\mathbb P^4)^*)$ that contains a pair of skew lines in a general hyperplane section of $X$.
\end{proposition}

Our aim is to better understand the fibers of $\pi$. Note that the fiber $\pi^{-1}(H)$ naturally corresponds to subschemes of the cubic surface $X\cap H$. Therefore, we would like to understand when a pair of lines on $S$ supports type (I)-(IV) subschemes, and how it relates to the singularities and singularities of the cubic surface.

Since every hyperplane section of a smooth cubic threefold is normal (Lemma \ref{Lemma_NormalHyperplane}), we will focus on the study of lines on normal cubic surfaces.

\begin{definition}\normalfont

Let $S\subseteq \mathbb P^3$ be a normal cubic surface. Recall that $H(\mathbb P^3)$ is the Hilbert scheme of a pair of skew lines on $\mathbb P^3$ (cf. Definition \ref{Def_Hm}). We denote $H(S)$ the scheme-theoretic intersection
$$\textup{Hilb}^{2n+2}(S)\cap H(\mathbb P^3).$$
\end{definition}
Informally, this is the "Hilbert scheme" of pairs of skew lines on the cubic surface $S$.

When $S$ is smooth, $H(S)$ consists of $216$ reduced points. This is because, for each of the 27 lines in a smooth cubic surface, there are 16 lines disjoint to it, so there are $27\times 16/2=216$ disjoint pairs.

When $S$ is "mild" singular, $H(S)$ is nonreduced but still has length 216. It is identified with a fiber of \eqref{eqn_piH(X)->P4*} and should contain type (II), (III), (IV) in a subtle way. We want to understand the relation between the scheme $H(S)$ and singularities on $S$.
\subsection{Lines on Normal Cubic Surfaces}
\begin{lemma}\label{Lemma_NormalHyperplane}
    Let $X$ be a smooth cubic threefold, then any hyperplane section $X\cap H$ is normal.
\end{lemma}
\begin{proof}
This is due to the classifications of cubic surfaces. According to Theorem 9.2.1 in \cite{Dolgachev}, a non-normal cubic surface is either cone over a singular cubic curve or projective equivalent to $t_0^2t_2+t_1^2t_3=0$, or 
$t_2t_0t_1+t_3t_0^2+t_1^3=0$. In either case, $S$ has to be singular along a line $L$. Now assume $S$ is the hyperplane section $t_4=0$, then $X$ has defining equation $$F_i(t_0,...,t_3)+t_4Q(t_0,...,t_4)=0,$$ 
with $F_i(t_0,...,t_3)$ the defining equation of $S_i$ and  $Q(t_0,...,t_4)$ a homogeneous quadric. Then by taking the partial derivatives and restricting them to the line, one finds that $X$ is singular at the intersection between the line $L$ and the quadric surface $Q(t_0,...,t_3,0)=0$, which contradicts that $X$ is smooth.
\end{proof}

Alternatively, one can use Serre's criterion for normality, as in \cite[Prop. 3.1]{BBF+}.

A normal cubic surface $S$ has only isolated singularities. The relation between the singularities and the number of lines on a cubic surface is a classical result. The following result is due to classifications of cubic surface \cite{Bruce-Wall} and \cite[section 9.2.2]{Dolgachev}).
 
\begin{lemma}\label{classicication of cubic surfaces}
Let $S$ be a normal cubic surface, then

(i) $S$ has at worse ADE singularities and has at most 27 lines, or

(ii) $S$ has an elliptic singularity and has a one-parameter family of lines.
\end{lemma}

 When $S$ has ADE singularities, there are at most 27 lines. The more singular $S$ is, the fewer number of lines $S$ has. For example, a cubic surface with an $A_1$ singularity has 21 lines, a cubic surface with three $A_2$ singularities has 3 lines (cf. Example \ref{Example_A1}, \ref{Example_3A2}), while a cubic surface with an $E_6$ singularity has only 1 line.

\begin{remark}\normalfont
Classically, Cayley \cite{Cayley} gave an account for the number "27" for cubic surfaces with ADE singularities. Each line passing through a singularity has a multiplicity of at least two, and the number "27" is interpreted as the number of lines counted with multiplicities. For example, a cubic surface with an $A_1$ singularity has 6 lines through the singularity and 15 lines away from the singularity (cf. Example \ref{Example_A1}). The 6 lines have multiplicity two, so the total number is $6\times 2+15=27$. As another example, a cubic surface with three $A_2$ singularities has only three lines (cf. Example \ref{Example_3A2}). Each line has multiplicity nine, so $3\times 9=27$.
\end{remark}

\begin{remark}\normalfont\label{remark_27lines}
One can interpret lines on a cubic surface $S$ with ADE singularities using the Hilbert scheme $\textup{Hilb}^{n+1}(S)$, then $\textup{Hilb}^{n+1}(S)$ has length 27. This is pointed out by Will Sawin. Let $\mathfrak{S}$ be the tautological subbundle on $Gr(2,4)$. Then a cubic surface $S$ determines a section to $\textup{Sym}^3\mathfrak{S}$ and its zero scheme is the Hilbert scheme of lines $\textup{Hilb}^{n+1}(S)$. It has dimension zero as long as $S$ has at worst ADE singularities by Lemma \ref{classicication of cubic surfaces}. Therefore, by the miracle flatness theorem, $\textup{Hilb}^{n+1}(S)$ has constant length 27. This idea is also studied by Tziolas in \cite{Tziolas} and \cite[Example 1.1 (b)]{Tziolas2}.   
\end{remark}

By Lemma \ref{classicication of cubic surfaces}, we conclude

\begin{corollary} \label{fin_line_cor} 
Let $S=X\cap H$ be a hyperplane section of a smooth cubic threefold. Then $S$ has only finitely many lines except when $H=T_pX$ is a tangent hyperplane of $X$ at an Eckardt point $p\in X$. In particular, when $X$ is general, all hyperplane sections $S$ of $X$ have only finitely many lines.
\end{corollary}

\subsection{A Pair of Skew Lines on Cubic Surface} Note that a normal cubic surface $S$ can be embedded into a smooth cubic threefold as a hyperplane section. So we can always assume that $S=X\cap H$ for some smooth cubic threefold. Typically, $S$ is singular exactly when $H$ is tangent to some point $p$ on $X$. Since every line $L$ on $X$ passing through $p$ is contained in $T_pX$, it is contained in $S$ as well.

We want to characterize $H(S)_{\textup{red}}$. This requires analyzing how the schemes of the four types are supported on the pair of lines of $S$ in different configurations.

\begin{proposition} \label{III_sur_prop}
  Let $L_1$ and $L_2$ be two lines on a normal cubic surface $S$.  \\
(1) If $L_1\cap L_2=\emptyset$, $L_1\cup L_2$ defines a type (I) subscheme of $S$;\\
(2) If $L_1\cap L_2=\{p\}$ and $p$ is a singularity of $S$, then  $L_1\cup L_2$ supports a unique type (III) subscheme of $S$. Conversely, the support of the embedded point of a type (III) subscheme of $S$ is always a singularity of $S$.
\end{proposition}

\begin{proof} 
Write $S=X\cap H$ as a hyperplane section. Then $S$ is singular at $p$ if and only if $H=T_pX$.
By Lemma \ref{typeIII_lemma}, $X$ contains a unique type (III) subscheme $Z$ supported on $L_1\cup L_2$ and $Z$ is contained in the tangent hyperplane $T_pX$. Therefore, $Z$ lies in $S$ iff $H=T_pX$ iff $p\in S$ is singular. 
\end{proof}

\begin{proposition}\label{Prop_cubicSurface12type}
Let $L\subseteq S$ be a line and $L\cap S^{\textup{sing}}\neq \emptyset$. Then there is a type (II) (resp. (IV)) subscheme $Z$ of $S$ supported on $L$ if and only if the torsion-free part of the normal sheaf $N_{L|S}$ is isomorphic to $\mathcal{O}_L$ (resp. $\mathcal{O}_L(1)$). Conversely, if $L$ supports a type (II)/(IV) subscheme of $S$, $L$ must pass through a singularity of $S$.
\end{proposition}
\begin{proof}
    Write $S=X\cap H$ as before. There exists a type (II)/(IV) subscheme $Z$ of $S$ supported on $L$ if and only if there exists such a type (II)/(IV) subscheme of $X$ contained in the hyperplane $H$. Now if $L$ as a line of $X$ is of the first type, then $L$ supports a type (II) subscheme $Z_{L,Q_p}$ of $X$ which is contained in $T_pX$ for some $p\in L$ (cf. Proposition \ref{quadric_surfaces_lemma}). In particular, $S$ contains $Z_{L,Q_p}$ if and only if $H=T_pX$ if and only if $S$ is singular at $p\in L$. In this case, the normal bundle $N_{L|Q}\cong \mathcal{O}_{L}$ embeds into $N_{L|S}$ as a subbundle. Similar to the type (IV) schemes.
\end{proof}
Compare to Definition \ref{Def_Line12type_cubic3fold}, we can call a line $L\subseteq S$ on cubic surface to be of the \textit{first type} (resp. \textit{second type}) if the torsion-free part of $N_{L|S}$ is $\mathcal{O}_L$ (resp. $\mathcal{O}_L(1)$). This is equivalent to that $L$ is a line of the first type (resp. second type) on the corresponding cubic threefold $X$.

Also, as a consequence of Proposition \ref{quadric_surfaces_lemma} and Proposition \ref{Prop_type2_P1family}, a line $L\subseteq S$ of the first type supports a unique type (II) subscheme, while a line of the second type supports (generically) two type (IV) subschemes.

\begin{example} \label{Example_A1}
\normalfont ($A_1$ singularity) Let $S_{A_1}$ be a cubic surface with a single singularity of type $A_1$. Then it has 21 lines, with 15 lines away from the nodes and 6 lines $L_1,...,L_6$ passing through the node. There are 120 disjoint pairs of lines corresponding to the type (I) scheme; Each $L_i$ supports a unique type (II) scheme structure, and there are 6 such lines. The union $L_i\cup L_j$ with $i\neq j$ supports a unique type (III) structure with the embedded point supported at the intersection point, and there are 15 such pairs. $$|H(S_{A_1})_{\textup{red}}|=120+6+15+0=141.$$
\end{example}

\begin{example} \label{Example_3A2}
\normalfont (3$A_2$ singularities) Let $S_{3A_2}$ be defined by $xyz=w^3$. Then it has only 3 lines $x=w=0,y=w=0,z=w=0$ with each line passing through two of the three $A_2$ singularities $[1,0,0,0]$, $[0,1,0,0]$, $[0,0,1,0]$. Each line supports type (IV) schemes in two different ways: The embedded point can be supported on either of the two singularities that the line passes through. $S_{3A_2}$ also contains type (III) schemes supported on every pair of the three lines with the embedded point supported at the intersection point. $$|H(S_{3A_2})_{\textup{red}}|=0+0+3+6=9.$$
\end{example}

\begin{example} \label{Example_elliptic}
\normalfont (elliptic singularity) Let $S_E$ be a cone over a smooth cubic curve $E$. Each line only supports type (IV) scheme in a unique way: The embedded point is supported at the cone point. Moreover, each pair of distinct lines intersects at the cone point, so their union supports a type (III) scheme. $$H(S_E)_{\textup{red}}\cong \textup{Sym}^2E.$$
\end{example}

\begin{figure}[h]
     \centering
     \begin{subfigure}[b]{0.2\textwidth}
         \centering
         \includegraphics[width=\textwidth]{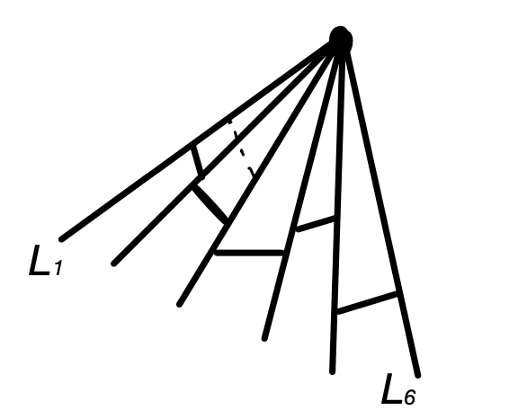}
         \caption*{$A_1$ singularity}
         \label{fig:A1}
     \end{subfigure}
     \hspace{2em}
     \begin{subfigure}[b]{0.2\textwidth}
         \centering
         \includegraphics[width=\textwidth]{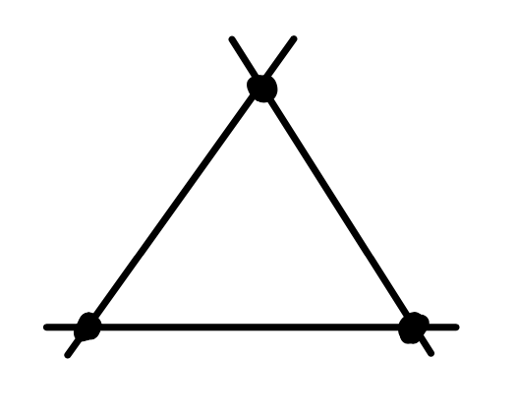}
         \caption*{3$A_2$ singularities}
         \label{fig:3A2}
     \end{subfigure}
      \hspace{2em}
   \begin{subfigure}[b]{0.2\textwidth}
         \centering
         \includegraphics[width=\textwidth]{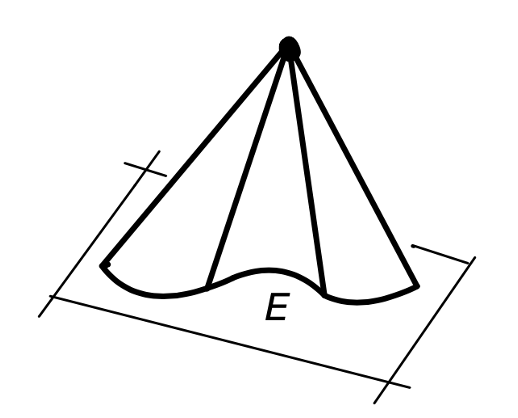}
         \caption*{elliptic singularity}
         \label{fig:cone}
     \end{subfigure}
    
        \caption{Examples of Singular Cubic Surfaces}
        \label{fig:singular_cubic_surfaces}
\end{figure}

Inspired by Remark \ref{remark_27lines}, we obtain the following result.
\begin{proposition}\label{prop_216}
    $H(S)$ is a zero-dimensional scheme and has a constant length of 216, as long as $S$ at worst ADE singularities. 
\end{proposition}
\begin{proof}
    This is due to $H(S)_{\textup{red}}$ being finite (Lemma \ref{classicication of cubic surfaces}, Proposition \ref{III_sur_prop}, \ref{Prop_cubicSurface12type}) and the miracle flatness theorem.
\end{proof}
\begin{proposition}\label{HYtoP4dual}
The morphism $$\pi:H(X)\to (\mathbb P^4)^*$$ is generically finite. The only positive dimensional fibers occur at tangent hyperplane $T_pX$ of an Eckardt point $p$ of $X$, and $\pi^{-1}(T_pX)$ is isomorphic to $\textup{Sym}^2E$, where $E$ is the elliptic curve corresponding to $p$.  
\end{proposition} 
\begin{proof}
It is a consequence of Proposition \ref{prop_216}, Corollary \ref{fin_line_cor} and Example \ref{Example_elliptic}.
\end{proof}

\begin{corollary}\label{Cor_ExtendGauss}
The extended Gauss map $\tilde{\mathcal{G}}:\textup{Bl}_0(\Theta)\to (\mathbb P^4)^*$ is generically finite. $\tilde{\mathcal{G}}^{-1}(t)$ has positive dimension if and only if
the hyperplane section $H_t\cap X$ is a cone over an elliptic curve $E$, and in which case $\tilde{\mathcal{G}}^{-1}(t)$ is isomorphic to $E$.
\end{corollary}

\begin{proof}Recall that there is a factorization \eqref{AJblowupDiagram} $$\Bl_{\Delta_F}(F\times F)\to \Bl_0(\Theta)\to (\mathbb P^4)^*.$$ 

As a consequence of Proposition \ref{HYtoP4dual} and Theorem \ref{main_theorem}, the positive dimensional fiber of $\tilde{\Phi}$ is $E\times E$ over tangent hyperplane at an Eckardt point.  We identify $E\times E\subseteq F\times F$ with its strict transform in the blow-up. It suffices to show that the image of positive dimensional fiber $\tilde{\Phi}^{-1}(t)\cong E\times E$ under the map $\tilde{\Psi}$ is isomorphic to $E$. 

The restriction of the Abel-Jacobi map \eqref{AJ} to $E\times E$ factors through
$$\Psi|_{E\times E}: E\times E\xrightarrow{f} E\xrightarrow{g} Alb(F)\cong JX,$$
where $f$ is given by $(p,q)\mapsto p-q$. The morphism $g$ is a closed immersion since it factors through $F\to Alb(F)$, which is a closed immersion \cite[Theorem 4]{Beauville2}. The last isomorphism is due to \cite[Theorem 11.19]{CG}. It follows that $\Psi(E\times E)\cong E$ is an elliptic curve, and by commutativity of \eqref{AJblowupDiagram}, the strict transform of $\Psi(E\times E)$ is contracted by $\tilde{\mathcal{G}}$.
\end{proof}
As a digression, we record the following consequence of Corollary \ref{Cor_ExtendGauss}. 

\begin{corollary}\label{Cor_GaussMapConst-on-E}
    Suppose a smooth cubic threefold $X$, then the Gauss map \eqref{Eqn_GaussMap} for the theta divisor of $JX$ is generically finite. The positive dimensional fibers are elliptic curves corresponding to Eckardt points on $X$.
\end{corollary}

\subsection{Normal Sheaf and Singularities}

As a consequence of Proposition \ref{Prop_cubicSurface12type}, if a line $L$ on $S$ is away from the singularities of $S$, then $L$ cannot support a type (II)/(IV) subschemes. On the other hand, the occurrence of singularity on $L$ changes the local geometry of $L$ near $S$ (cf. Proposition \ref{Prop_cubicSurface12type}). One may ask whether the singularities on $S$ on a line $L$ determine the type of the line.

In general, this is false. As an example, the cubic surface defined by $F=x_0x_1x_2+x_2x_3^2+x_3x_1^2$ is normal with an $A_4$ singularity at $[1:0:0:0]$ (and an $A_1$ singularity at $[0:0:1:0]$). Both of the lines $x_2=x_3=0$, $x_1=x_2=0$ only pass through the $A_4$ singularity. However, by computing the dual map along the lines, we conclude that one of the lines is of the first type and the other is of the second type.

However, we do have some deterministic results.

\begin{proposition}

(1) If a line $L\subseteq S$ passes through only one singularity which has type $A_1$, then $L$ is of the first type.

(2) If $L$ passes through more than one singularity, then $L$ is of the second type.
\end{proposition}
\begin{proof}
For (1), assume the line $L$ is given by $x_2=x_3=0$ and the cubic surface has an equation
$$x_0Q(x_1,x_2,x_3)+C(x_1,x_2,x_3)=0,$$
with the cubic $C=x_1^2(ax_2+bx_3)+\cdots$ and quadric $Q=x_1(cx_2+dx_3)+\cdots$ intersecting transversely at $6$ points, which guarantees that $[1:0:0:0]$ is an $A_1$ singularity. The dual map along $L$ is $\mathcal{D}|_L=[0:0:cx_0+ax_1:dx_0+bx_1]$.  The transversality condition implies that the two linear forms are linearly independent, so it corresponds to the dual map of a smooth quadric surface along $L$.

For (2), we regard $S$ as a hyperplane section of a smooth cubic threefold $X$, then use the fact that the dual map \eqref{DualMap} on $X$ along $L$ is 1-to-1 (resp. 2-to-1) when $L$ is of the first type (resp. second type).
\end{proof}

\section{Relations to Other Work}\label{Section_5.HS_Modular}
In this section, we will discuss some relations to the Bridgeland stable moduli spaces studied in \cite{APR} and \cite{BBF+}.

In \cite{APR}, the author studied the moduli space $\mathcal{M}_{\sigma}(w)$ of Bridgeland stable objects in the Kuznetsov component with Chern character $w=H-\frac12H^2+\frac13H^3$ for a smooth cubic threefold $X$. Let $S$ denote a hyperplane section of $X$. The moduli space $\mathcal{M}_{\sigma}(w)$ parameterizes the following two objects:

(1) $\mathcal{O}_S(D)$, a reflexive sheaf or rank one associated with certain Weil divisor $D$ on $S$ (when $S$ is general, $D=L_1-L_2$ for a pair of skew lines $L_1,L_2$ on $S$);

(2) $I_{p|S}$, the ideal sheaf of a point in $S=X\cap H$, where $H$ is the tangent hyperplane section at $p$.

In both cases, the stable object is contained in a unique hyperplane section, so there is a natural projection
\begin{equation}
   \mathcal{M}_{\sigma}(w)\to (\mathbb P^4)^*. \label{MSigmaProjEqn}
\end{equation}

\begin{proposition}
The projection \eqref{MSigmaProjEqn} is generically finite, and its only positive dimensional fibers are elliptic curves corresponding to the Eckardt points on $X$.
\end{proposition}
\begin{proof}
It is shown in section 3.3 \cite{APR} that the $\mathcal{M}_{\sigma}(w)$ is isomorphic to the moduli space $\mathcal{M}_G(\kappa)$ of Gieseker-stable sheaves with Chern character $\kappa=(3,-H,\frac12H^2,-\frac16H^3)$ studied in \cite{BBF+}. According to Lemma 7.5 of \cite{BBF+}, there is an isomorphism $\mathcal{M}_G(\kappa)\cong \textup{Bl}_0(\Theta)$. Via the isomorphisms, the projection \eqref{MSigmaProjEqn} coincides with the extended Gauss map \eqref{AJblowupDiagram} $$\textup{Bl}_0(\Theta)\to (\mathbb P^4)^*,$$ 
since they agree on a general point. Now the argument follows from Corollary \ref{Cor_ExtendGauss}.
\end{proof}
\subsection{A Modular Interpretation of Extended Abel-Jacobi Map}
Recall that $\widetilde{H(X)}\cong \textup{Bl}_{\Delta_F}(F\times F)$ is a branched double cover of the Hilbert scheme of a pair of skew lines on $X$.  Denote $\widetilde{AJ}:\widetilde{H(X)}\to JX$ the composition of the blow-up map and $F\times F\to JX$ as in \eqref{diagram_blowupAJ}. For the Bridgeland moduli space $\mathcal{M}_{\sigma}(w)$, there is also an Abel-Jacobi map 
\begin{equation}\label{eqn_AJM_{sigma}}
    AJ: \mathcal{M}_{\sigma}(w)\to JX
\end{equation}
by sending $\mathcal{O}_S(D)$ to $\int_{D_0}^{D}$, for some fixed divisor $D_0$.

Regarding $\mathcal{M}_{\sigma}(w)$ as the blowup of the theta divisor, then according to  \cite[Proposition 2.1]{APR}, the exceptional divisor $X$ parameterizes ideal sheaves $I_{p|S}$ of singular points on hyperplane sections $S$ of $X$. The complement of exceptional divisor parameterizes coherent sheaves $\mathcal{O}_S(D)$ with $D$ being a certain Weil divisor on the hyperplane section $S$.

We propose the following modular interpretation of the extended Abel-Jacobi map \eqref{diagram_blowupAJ}.

\begin{proposition} \label{App_Prop_Msigma-factor}
The diagram \ref{diagram_blowupAJ} is identified with the following commutative diagram up to adding a constant.

\begin{figure}[ht]
    \centering
\begin{equation}\label{H(X)Msigma_diagram}
\begin{tikzcd}
\widetilde{H(X)}\arrow[dr, "\widetilde{AJ}"'] \arrow[r,"\psi"] &  \mathcal{M}_{\sigma}(w) \arrow[d,"AJ"]\\
  & JX
\end{tikzcd}
\end{equation}
\end{figure}
Suppose $(L_1,L_2)\in \widetilde{H(X)}$ is a pair of skew lines on $S=X\cap \textup{Span}(L_1,L_2)$ and disjoint from the singularities of $S$, then $\psi$ sends $(L_1,L_2)$ to $\mathcal{O}_S(L_1-L_2)$, where $S$ is the hyperplane section defined by $\textup{Span}(L_1,L_2)$. $\psi$ sends a scheme of type (II) or (IV) to the ideal sheaf $I_p$, where $p$ is the unique point determined by $Z$ defined in Corollary \ref{Cor_Q1Answer}. 
\end{proposition}

\begin{proof}
The map $\psi: (L_1,L_2)\mapsto \mathcal{O}_S(L_1-L_2)$ is generically 6-to-1 and agrees with $\tilde{\Psi}$, because on a smooth hyperplane section, $[L_1]-[L_2]$ is a vanishing cycle and can be represented by the difference of two skew lines in six ways. In particular, $\psi$ extends to a morphism and coincides with $\tilde{\Psi}$. The Abel-Jacobi map $\widetilde{AJ}$ factor through $\psi$ by definition up to translating by an image of $D_0$. 

The rest of the claim is due to the knowledge of stable objects in $\mathcal{M}_{\sigma}(w)$. If $L_1$ and $L_2$ are skew and disjoint from singularities of $S$, then the intersection pairing on $S$ for $L_1$ and $L_2$ are the same as the smooth case. In particular, $L_1-L_2$ has self-intersection number $-2$ and $\mathcal{O}_S(L_1-L_2)\in \mathcal{M}_{\sigma}(w)$.
\end{proof}

In general, it is not explicitly known what the divisor $D$ is on a singular hyperplane section $S$ such that $\mathcal{O}_S(D)\in \mathcal{M}_{\sigma}(w)$. This prevents us from claiming Proposition \ref{App_Prop_Msigma-factor} for type (I) scheme passing through singularities and the type (III) (cf. Proposition \ref{III_sur_prop}). In other words, it is not known what the representative of the flat limit of $\mathcal{O}_S(L_1-L_2)$ as two skew lines specialize to intersect at one point, or when the hyperplane section becomes singular and at least one line pass through the singularity.

We conjecture the modular interpretation of $\psi$ in Proposition \ref{App_Prop_Msigma-factor} holds everywhere on $\widetilde{H(X)}$:

\begin{conjecture}\label{Conj_typeIII}
Let $p\in \mathcal{M}_{\sigma}(w)$ be a closed point not on the exceptional divisor,  then it parameterizes stable object of the form $\mathcal{O}_S(L_1-L_2)$, where $L_1, L_2$ are two lines on $S$ (can be singular) and lie in one of the two cases:

(i) $L_1, L_2$ disjoint (and may pass through singularities);

(ii) $L_1$ and $L_2$ intersect at one point $p$, which is a singularity of $p$. 

\end{conjecture}

Note that if the Conjecture \ref{Conj_typeIII} holds, we can classify all stable objects in $\mathcal{M}_{\sigma}(w)$. The proof requires careful study of the intersection pairing of Weil divisors on the normal cubic surfaces \cite{Sakai}. We hope to study it in future research.

\subsection{A Question in Cubic Fourfold}
Let $Y$ be a cubic fourfold. Let $H(Y)$ be the component of Hilbert scheme of $Y$ that contains a pair of skew lines. Our next goal is to

\begin{problem}
Characterize the Hilbert scheme of a pair of skew lines $H(Y)$ of the cubic fourfold.
\end{problem}

We suspect that $H(Y)$ is also smooth and arises from $\textup{Sym}^2F(Y)$ by a two-step successive blow-up, where $F(Y)$ be the Fano variety of lines of $Y$. If this is true, this allows us to study a hyperkahler 8-fold associated with $Y$. 

According to \cite{LLSvS}, if $Y$ does not contain a plane, there is a hyperkahler 8-fold $W$ of $\textup{K3}^{[n]}$-type associated with $Y$ arising from Hilbert schemes $\mathcal{C}_Y$ of twisted cubics. In fact, there is two-step contraction $\mathcal{C}_Y\xrightarrow{p_1} W'\xrightarrow{p_2} W$ where $p_1$ is a $\mathbb P^2$ bundle and $p_2$ is the blow-up of a 4-dimensional subvariety of $W$ which is isomorphic to $Y$. 

Suppose we solved the problem as predicted. Let $\widetilde{H(Y)}$ be the natural double cover of $H(Y)$ by giving an order to a pair of skew lines. Then there is a dominant rational map of degree 6
$$\phi:\widetilde{H(Y)}\dashrightarrow W$$
extending Voisin's rational map \cite[Prop. 4.8]{VoisinRatMap} whose restriction to a general hyperplane section $X$ extends to the morphism \eqref{diagram_blowupAJ}
$$\textup{Bl}_{\Delta}(F\times F)\to \textup{Bl}_0(\Theta),$$
which arises from extending the Abel-Jacobi map of the cubic threefold $X$. It turns out that $\textup{Bl}_0(\Theta)$ is a Lagrangian subvariety of the hyperkahler variety $W$ (cf. \cite{Evgeny}).  So by deforming the hyperplane section $X$, this becomes a Lagrangian cover of $W$ in the sense of \cite{Voisin22_Lagrangian}.

Recently, Flapan, Macri, O'Grady, and Sacca found a different Lagrangian of $W$ and is of general type \cite{FMOS_I,FMOS_II}, which arises as a component of the fixed locus of an antisymplectic involution $\tau:W\to W$ that is defined lattice-theoretically using the Beauville–Bogomolov–Fujiki quadratic form and the Torelli type theorem lifting to the variety $Z$. In fact, they find the fixed locus $Fix(\tau)$ has two components $Y_1$ and $Y_2$, with $Y_1\cong Y$ isomorphic to the cubic fourfold, and $Y_2$ is of the general type. According to Beauville, these are Lagrangian subvarieties. O'Grady conjectured the following.

\begin{conjecture} (O'Grady) the Lagrangian subvariety $Y_2$ of the general type deforms in a family that forms a Lagrangian cover of $W$.   
\end{conjecture}

From our perspective, the involution $\tau$ is invariant on each  $\textup{Bl}_0(\Theta)$ supported on each hyperplane section, and the involution is induced from $JX\to JX$ by sending $x\mapsto -x$. So in particular, the fixed locus consists of those 2-torsion points that lie on theta divisor of a general hyperplane section $X$ of $Y$. Over zero point $0\in \Theta$ of $JX$, the exceptional divisor is isomorphic to $X$ \cite{Beauville}, and piles up to the Fano component $Y_1\cong Y$ as hyperplane deforms. The nonzero 2-torsion points supported on $\Theta$ should pile up to the general type component $Y_2$. We want to understand from the perspective of the map $\phi$, how to understand the component $Y_2$. Indeed, the involution lifts to $\tilde{\tau}$ on $\widetilde{H(Y)}$ by permitting the order of two skew lines. In particular, the fixed locus of $\tilde{\tau}$ are the subschemes of $Y$ of type II and IV that are generically nonreduced and their image is $Y_1$. The preimage $\tilde{Y}_2$ of $Y_2$ on the other hand, is not fixed by $\tilde{\tau}$. We want to understand the 4-dimensional subvariety $\tilde{Y}_2$ of $\widetilde{H(Y)}$, and think about if there is a geometric reason it can be deformed.

\bibliographystyle{alpha}
\bibliography{bibfile}

\end{document}